\documentclass[paper]{article}
\usepackage{amsmath,amssymb,a4wide}
\usepackage{color,graphicx,epstopdf}
\usepackage{caption}

\newcommand{\bu}{\boldsymbol u}

\newcommand{\bff}{\boldsymbol f}

\newcounter{remark}
\def\theremark {\arabic{remark}}

\newtheorem{Proof}{Proof}

\newenvironment{proof}{\begin{Proof}\rm}{\hfill $\Box$ \end{Proof}}
\usepackage{hyperref}
\newtheorem{definition}{Definition}
\newtheorem{proposition}{Proposition}

\title{Can Neural Networks learn Finite Elements?}
\author{
 Julia Novo\thanks{ Departamento de
Matem\'aticas, Universidad Aut\'onoma de Madrid, Spain. Research is supported
by grant PID2022-136550NB-I00  funded by MCIN/AEI/
10.13039/501100011033 and by ERDF A way of making Europe, by
the European Union. (julia.novo@uam.es)} \and Eduardo Terr\'es \thanks{(eduardo.terrescaballero@outlook.com)}}
\date{\today}
\begin{document}
\maketitle
\abstract{The aim of this note is to construct a neural network for which the linear finite element approximation of a simple one dimensional boundary value problem is a minimum of the cost function to find out if the neural network is able to reproduce the finite element approximation. The deepest goal is to shed some light on the problems one encounters when trying to use neural networks to approximate partial differential equations}
\bigskip


\section{Introduction} For a brief introduction to neural networks and their applications to the numerical approximation
of partial differential equations we refer the reader to reviews \cite{siam_rev1}, \cite{siam_rev2}. For
the origins of the finite element method and its influence in modern computing we refer to review \cite{GanderWanner2012}.

The literature concerning the use of neural networks to approach partial differential equations has increased {significantly} in recent years. Most of the results are heuristic and the ability of neural networks
to perform better than classical methods is not always clear. With this short note we would like to shed some light
on the difficulties one encounters {when} trying to {define} a neural network {to} approach a given problem better than  classical finite element methods. To this end, we choose a very simple model problem, that nevertheless allows us to
understand some general characteristics of neural networks.

When choosing a neural network to approach a problem one has to decide {its structure} and activation function. The structure of the network, a priori unknown, has a great impact on the obtained results. The use of different activation functions also leads to different approaches.
Another important question is the theoretical ability of a neural network to approximate a function arbitrarily well, see \cite{Pinkus}.

In this note, following \cite{He_et_al}, we choose the structure of the neural network in such a way that for certain values of the weights and biases the output of the network is the linear finite element approximation to the model problem. The activation function is the so called ReLU.
{With this structure, certain theoretical capabilities are guaranteed, given that} piecewise linear approximations over a certain partition can be exactly represented by the neural network.
In particular, the response of the network could be the finite element approximation itself, with well-known approximation properties.
However, we observed a first problem. Even in this case in which the structure of the network is {known}, the result of the optimization problem that finds the values of the weights and biases can lead to an approximation far away from the linear finite element approximation that
a priori gives a minimum of the optimization problem. The idea behind this fact is very simple{: the problem we are trying to solve is highly over-determined, thus it is difficult to find the optimal value.} On a second try, we add some {initial} information to the neural network concerning the values of weights and biases to lead {somehow} the answer toward the output we know is a minimum of the problem.
As we show in the numerical experiments, we consider several levels in the procedure of adding information.
However, if we know how to compute the linear finite element approximation in a deterministic way it probably does not make sense to use neural networks to get this approximation.
So, the question is, can we find a balance between the deterministic information we supply and the free values we keep (output values of the optimization problem) so that the resulting approximation has the same or, if possible, better approximation capabilities than the standard finite element approximation?
This question can perhaps be re-formulated in the following way:  can physics-informed{, or similarly defined,} neural networks, \cite{mis_mo}, give better approximations than standard methods?
Our conclusion is that {treating neural networks as black boxes, without providing sufficient domain-specific insights does not offer guarantees for success.}
At least with the procedure we followed, in our simple model problem we do not obtain, in general, better approximations than the finite element method.
Moreover, we  only obtain neural network approximations close to those of the finite element method
in our last try, in which we have fixed so many weights and biases, that the response of the neural network, regardless of the output values of the optimization problem, is always a continuous piecewise linear approximation with respect to the partition in which the linear finite element approximation is defined.

This is not the first reference to pose the above question. In \cite{schonlieb_et_al}, the authors make a deep comparison between PINNs and finite elements.
Their study suggests that for certain classes of PDEs for which classical methods are applicable, PINNs do not produce better results.
As the authors of \cite{schonlieb_et_al} state, PINNs could however be efficient in high-dimensional problems for which classical techniques are prohibitively expensive and when combining PDEs and data.

The outline of the paper is as follows.
In Section 2 we introduce the model problem (based on a one dimensional convection-diffusion equation) and describe the neural networks that mimic finite elements.
In Section 3 we present a series of numerical experiments in which we distinguish two cases, corresponding to diffusion dominated and convection dominated examples.
In the convection dominated case, instead of the standard linear finite element approximation we try to reproduce the stabilized streamline-upwind Petrov-Galerkin (SUPG) linear finite element approximation
\cite{BH}  for which a {different cost function} from that of the standard method is used.
\section{Model problem. Neural networks that mimic finite elements}
As a model problem we consider the following steady convection-diffusion problem in one dimension. Find $u\in C^2(0,1)\cap C[0,1]$ such that
\begin{eqnarray}\label{eq_model}
    &&-\epsilon \, u''+u'=1,\quad x\in(0,1),\\
    &&u(0)=u(1)=0,\nonumber
\end{eqnarray}
where $\epsilon$ is a positive parameter. The weak solution of \eqref{eq_model} is: find $u\in H_0^1(0,1)$ such that
\begin{eqnarray}\label{eq_model2}
    \epsilon \, (u',v')+(u',v)=(1,v),\quad \forall v\in H_0^1(0,1),
\end{eqnarray}
where $(\cdot,\cdot)$ represents the $L^2(0,1)$ inner product and we have used standard notation for the Sobolev space $H_0^1(0,1)$.
For any integer $N>0$ we consider a partition
\begin{equation}\label{partition}
    \tau_N=\left\{0=x_0<x_1<\ldots<x_N=1\right\},
\end{equation}
of $[0,1]$ and denote by $h_i=x_{i+1}-x_i$. We denote by $V_h$ the linear finite element space of continuous functions
that are piecewise linear with respect to the partition $\tau_N$ and satisfy the homogeneous Dirichlet boundary conditions of the problem.
Then, the linear finite element approximation to the model problem \eqref{eq_model} satisfies the weak form \eqref{eq_model2} over $V_h$: find $u_h\in V_h$ such that
\begin{eqnarray}\label{eq_gal}
    \epsilon \, (u_h',v_h')+(u_h',v_h)=(1,v_h),\quad \forall v_h\in V_h.
\end{eqnarray}
With the standard nodal Lagrange basis functions, $\phi_i\in V_h$, defined by $\phi_i(x_j)=\delta_{i,j}$ we can write
$u_h(x)=\sum_{i=1}^{N-1}u_h(x_i)\phi_i(x)$. Then, $u_h$ is unique and its vector of coefficients
$\bu=\left[u_h(x_1),\ldots,u_h(x_{N-1})\right]^T$ is the unique solution of the linear system
$A \bu=\bff$, with $A=\epsilon K+C$ and $\bff=[f_1,\ldots,f_{N-1}]^T$ with $f_i=(h_{i-1}+h_i)/2$. The
matrices
$K$ and $C$ are tridiagonal and its elements are:

\begin{alignat}{3}
        &k_{i,i-1}=-\frac{1}{h_{i-1}}, \quad
        & &k_{i,i}=\frac{1}{h_{i-1}}+\frac{1}{h_i}, \quad
        & &k_{i,i+1}=-\frac{1}{h_i}\label{coef}
        \\
        &c_{i,i-1}=-\frac{1}{2},
        & &c_{i,i}=0,
        & &c_{i,i+1}=\frac{1}{2}.\notag
\end{alignat}
Let us denote by $\sigma$ the ReLU activation function defined as $\sigma(x)=x$ for $x\ge 0$ and $\sigma(x)=0$ for $x<0$. Then, following \cite{He_et_al}, we can write
$$
    \phi_i(x)=\sigma \left( \frac{x - x_{i-1}}{h_{i-1}} \right)
    - \sigma \left( (x - x_{i}) \left( \frac{1}{h_{i-1}} + \frac{1}{h_{i}} \right) \right)
    + \sigma \left( \frac{x - x_{i+1}}{h_{i}} \right).
$$
Following the notation in \cite{siam_rev1}, any of the basis functions can be written as a neural network
of the form
\begin{equation}\label{laphii}
    \phi_i(x)=W_i^{[3]}\sigma(W_i^{[2]}x+b_i^{[2]}), \quad i=1,\ldots,N-1,
\end{equation}
where
\begin{align}
    W_{i}^{[2]}                                                                          & =
    \begin{pmatrix}
        \frac{1}{h_{i-1}}                   \\
        \frac{1}{h_{i-1}} + \frac{1}{h_{i}} \\
        \frac{1}{h_{i}}
    \end{pmatrix}, \! &
    b_{i}^{[2]}                                                                          & =
    \begin{pmatrix}
        \frac{-x_{i-1}}{h_{i-1}}                                                \\
        - x_{i} \left( \frac{1}{h_{i-1}} + \frac{1}{h_{i}} \right) \label{W2b2} \\
        \frac{-x_{i+1}}{h_{i}}
    \end{pmatrix}, \\\label{W3}
    W_{i}^{[3]}                                                                          & =
    \begin{pmatrix}
        1, \, -1, \, 1
    \end{pmatrix}.
\end{align}
We now define a neural network that mimics the finite element approximation.
\begin{definition}  Let $F: {\Bbb R} \to {\Bbb R}$, be a neural network with three layers
    {where layer 1 is the Input layer and layer 3 is the Output layer,
    and layers 1, 2 and 3 have 1, $3(N-1)$ and 1 neurons, respectively.}
    More precisely, for $W^{[2]} \in {\Bbb R}^{3(N-1) \times 1}$, $b^{[2]} \in {\Bbb R}^{3(N-1)}$ and $W^{[3]} \in {\Bbb R}^{1 \times 3(N-1)}$,
    \begin{equation*}
        F (x)
        = W^{[3]} \sigma (W^{[2]} x + b^{[2]}).
    \end{equation*}

\end{definition}
\begin{proposition}\label{prop1}
    Let us define $W^{[2]}, b^{[2]}$ and $W^{[3]}$ by:
    \begin{align*}
        W^{[2]} =
        \begin{pmatrix}
            W_{1}^{[2]} \\
            \vdots      \\
            W_{i}^{[2]} \\
            \vdots      \\
            W_{N-1}^{[2]}
        \end{pmatrix}, \;
        b^{[2]} =
        \begin{pmatrix}
            b_{1}^{[2]} \\
            \vdots      \\
            b_{i}^{[2]} \\
            \vdots      \\
            b_{N-1}^{[2]}
        \end{pmatrix}, \;
        W^{[3]} =
        \begin{pmatrix}
            u_h(x_1) \, W_{1}^{[3]}   \\
            \vdots                    \\
            u_h(x_{i}) \, W_{i}^{[3]} \\
            \vdots                    \\
            u_h(x_{N-1}) \, W_{N-1}^{[3]}
        \end{pmatrix} ^ T
    \end{align*}
    where $W_i^{[2]}, b_i^{[2]}$ and $W_i^{[3]}$ are defined in \eqref{W2b2} and \eqref{W3}.
    Then $F \equiv u_h$.
\end{proposition}
\begin{proof}
    The proof is immediate using the definition of  $u_h(x)=\sum_{i=1}^{N-1}u_h(x_i)\phi_i(x)$, the expression of the basis functions as neural networks \eqref{laphii}
    and  Definition 1.
\end{proof}
Proposition 1 says that the neural network $F$ mimics the finite element approximation in the sense that there exist values of the weights and biases for which we recover the finite element approximation. We observe that
there are different values of weights and biases that satisfy the property $F \equiv u_h$ since $\sigma(\alpha x)=\alpha \sigma(x)$ for any $\alpha>0$.

Once the structure of the neural network is chosen (depending only on the value $N$), we still need to define a cost function.
{Since we want the network to mimic the finite element approximation, we define the cost function as follows}
\begin{eqnarray}
    {\rm Cost}&=&\sum_{i=1}^{N-1}\left(-\epsilon \left[-\frac{1}{h_{i-1}}F(x_{i-1})+\left(\frac{1}{h_{i-1}}+\frac{1}{h_i}\right)F(x_i)-\frac{1}{h_i}F(x_{i+1})\right]\right.\nonumber\\
    &&\left.+\frac{1}{2}\left[F(x_{i+1})-F(x_{i-1})\right]-\frac{1}{2}(h_{i-1}+h_i)\right)^2+F(0)^2+F(1)^2.
    \label{cost_1}
\end{eqnarray}
It is easy to check that choosing $F(x_i)=u_h(x_i),$ $i=0,\ldots,N$ then ${\rm Cost=0}$. Then, according
to Proposition 1, there exist weights and biases for which there is a minimum of the optimization problem.

It is well-known that in the convection dominated regime (in our case $\epsilon$ small in \eqref{eq_model})
the standard finite element approximation produces spurious oscillations unless the mesh size is {sufficiently small}. Stabilized approximations are ussually applied to mitigate this problem. The so called
SUPG  method is one of the most popular finite element stabilizations, \cite{BH}. Then, in the convection dominated regime instead of \eqref{cost_1} we consider
\begin{eqnarray}
    {\rm Cost}&=&\sum_{i=1}^{N-1}\left(-\epsilon \left[-\frac{1}{h_{i-1}}F(x_{i-1})+\left(\frac{1}{h_{i-1}}+\frac{1}{h_i}\right)F(x_i)-\frac{1}{h_i}F(x_{i+1})\right]\right.\nonumber\\
    &&\left.+\left[F(x_{i})-F(x_{i-1})\right]-\frac{1}{2}(h_{i-1}+h_i)\right)^2+F(0)^2+F(1)^2.
    \label{cost_2}
\end{eqnarray}
It is easy to check that  $F(x_i)=u_h^s(x_i),$ $i=0,\ldots,N$, with $u_h^s$ the linear finite element SUPG approximation to problem  \eqref{eq_model} with stabilization parameter $\delta_i=h_{i-1}/2$ in $(x_{i-1},x_i)$, satisfies ${\rm Cost}=0$.
\section{Numerical experiments}
We consider the neural network of Definition 1. In the numerical experiments we use the gradient descent method with back propagation. In the sequel for simplicity we consider a uniform partition \eqref{partition} for different values of $N$. We denote by $\eta$, the learning rate, and $N_{\rm iter}$ the
maximum number of iterations in the gradient descent method. We have tried to add a regularization term to the Cost function of the form $\beta \|w\|^2$, where $\|w\|^2$ represents the $L^2$ norm of the vector that contains all the weights and biases. However, most of the experiments are done with $\beta=0$, since a small value
of $\beta$ does not affect the results and by increasing the value of $\beta$ we get worse results. We have carried out two types of experiments. In the diffusion dominated regime we consider the cost function \eqref{cost_1}. In the convection dominated regime we consider the cost function \eqref{cost_2}.
\subsection{Diffusion dominated regime}
\begin{figure}[h]
    \begin{center}
        \includegraphics[height=4.5truecm]{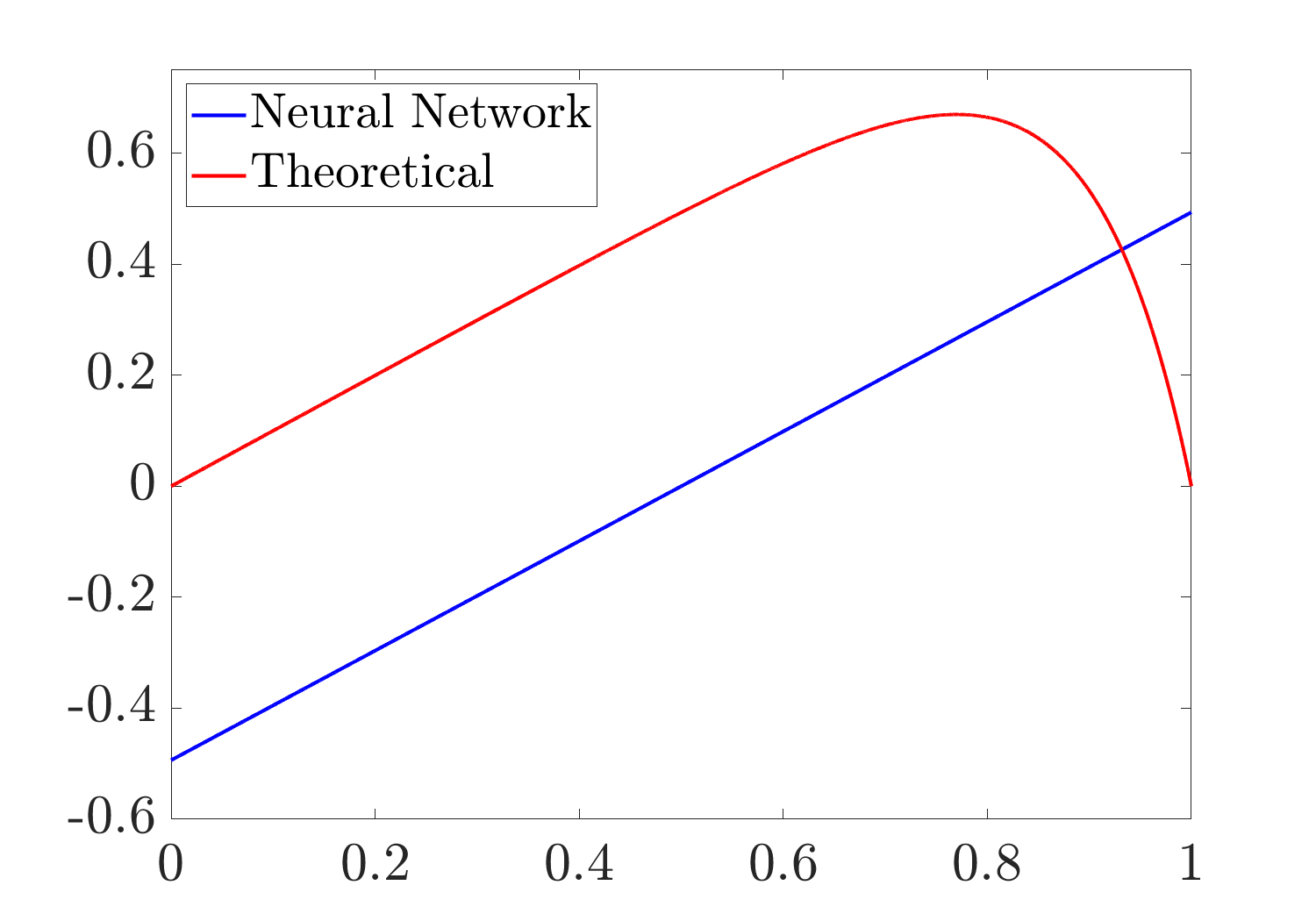}\
        \includegraphics[height=4.5truecm]{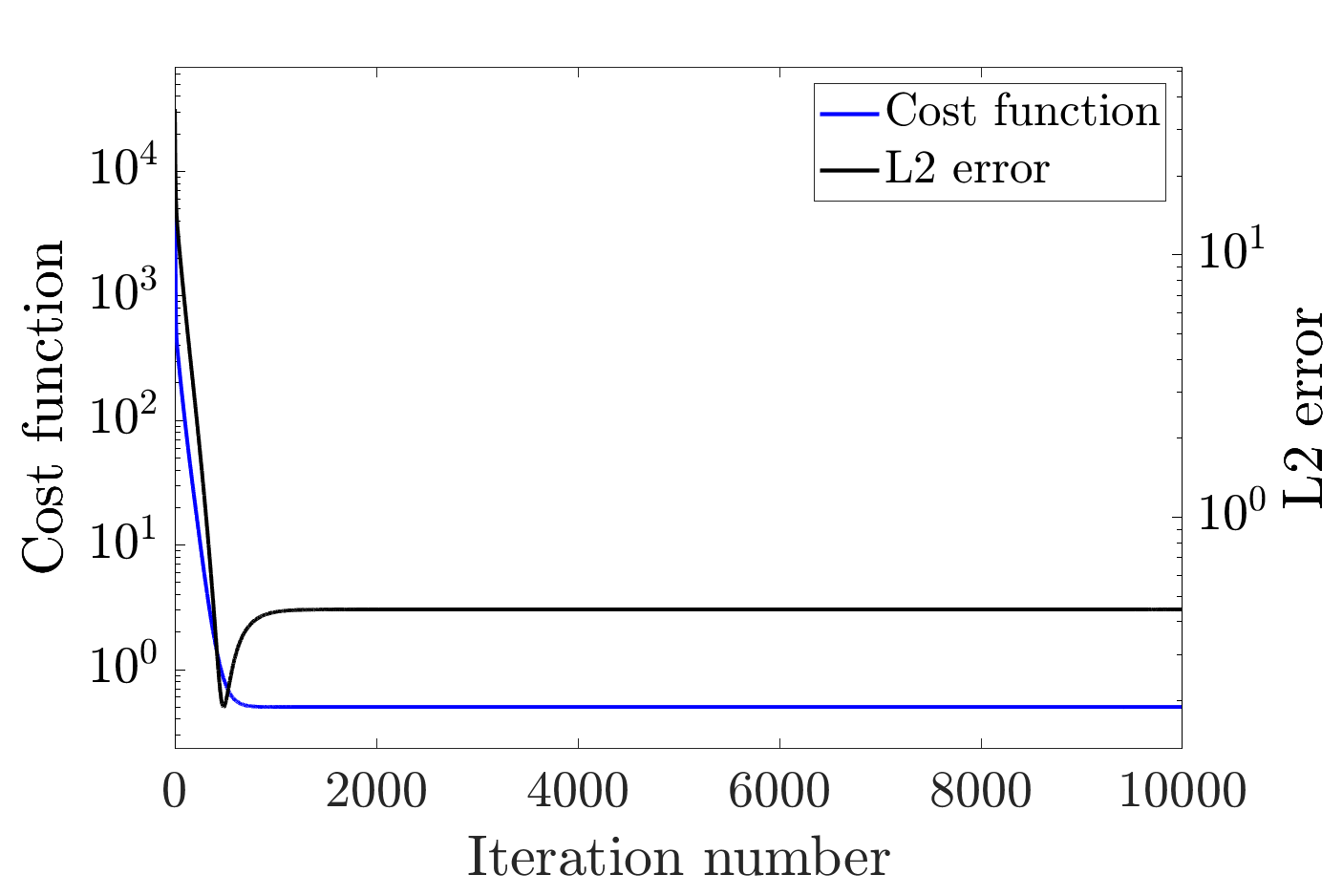}
        \begin{caption}{\label{fig1}$N=40$, $N_{\rm iter}=10^4$, $\eta=10^{-4}$, $\beta=10^{-4}$.}
        \end{caption}
    \end{center}
\end{figure}
We fix the value of $\epsilon=0.1$ along this section.
In a first experiment all the parameters, weights and biases, in the
neural network are free and we initialize them with random values. We observe that, while the finite element approximation has $N-1$ free parameters, our neural network has $9(N-1)$. In Figure \ref{fig1} we have plotted, on the left, the neural network approximation in blue and the exact solution in red and, on the right, the
values of the cost function in blue and the $L^2$ error between the exact solution and the neural network approximation in black. Let us observe that we {use different scales on the left and right axis} for the picture on the right. The left axis is the scale for the cost function while the right axis is the scale for the error. After less than 1000 iterations the value of the cost function stabilizes at around 0.5.
With this procedure the neural network is not able to reach a minimum of the cost function.
{The main issue the network faces here is solving a highly over-determined problem}.
We have tried different values of $N$, $N_{\rm iter}$, $\eta$ and $\beta$ without improving (even worsening) the results shown in Figure \ref{fig1}.

{
    In a second experiment, we try to add some initial information to the neural network with the aim of improving the results obtained in the previous failed attempt.
    Along this line, we initialize $W^{[2]}$ and $b^{[2]}$ with the values \eqref{W2b2} and, as before, we use random values for $W^{[3]}$.
    After initialization, the neural network is free to adjust the values of $W^{[2]}$ and $b^{[2]}$ in the
    pursuit of a configuration that minimizes the cost function and, potentially, outperforms the finite element solution.
}
Let us observe that we can write the neural network as $F(x)=\sum_{i=1}^{3(N-1)}w^3_i\sigma (w^2_ix+b^2_i)$ which means that $F(x)$ is a piecewise linear function over a mesh with nodes $-b^2_i/w^2_i.$ If all the nodes are positive then $F(0)=0$ while $F(1)=\sum_{i=1}^{3(N-1)}w^3_i\sigma (w^2_i+b^2_i)$. In the experiments we have observed that $F(0)=0$ while $F(1)$ is small (not strongly imposed).
\begin{figure}[h]
    \begin{center}
        \includegraphics[height=4.5truecm]{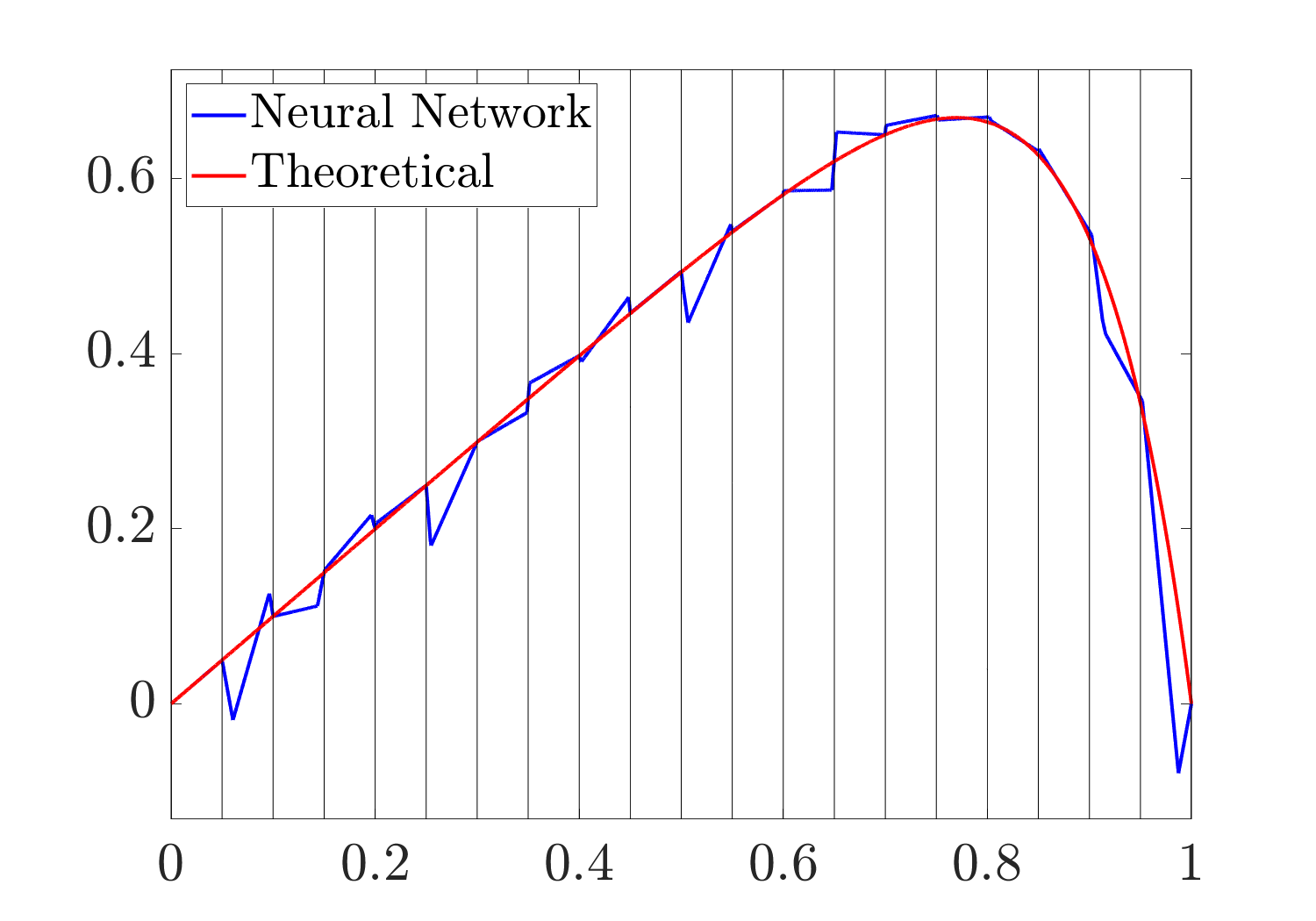}\
        \includegraphics[height=4.5truecm]{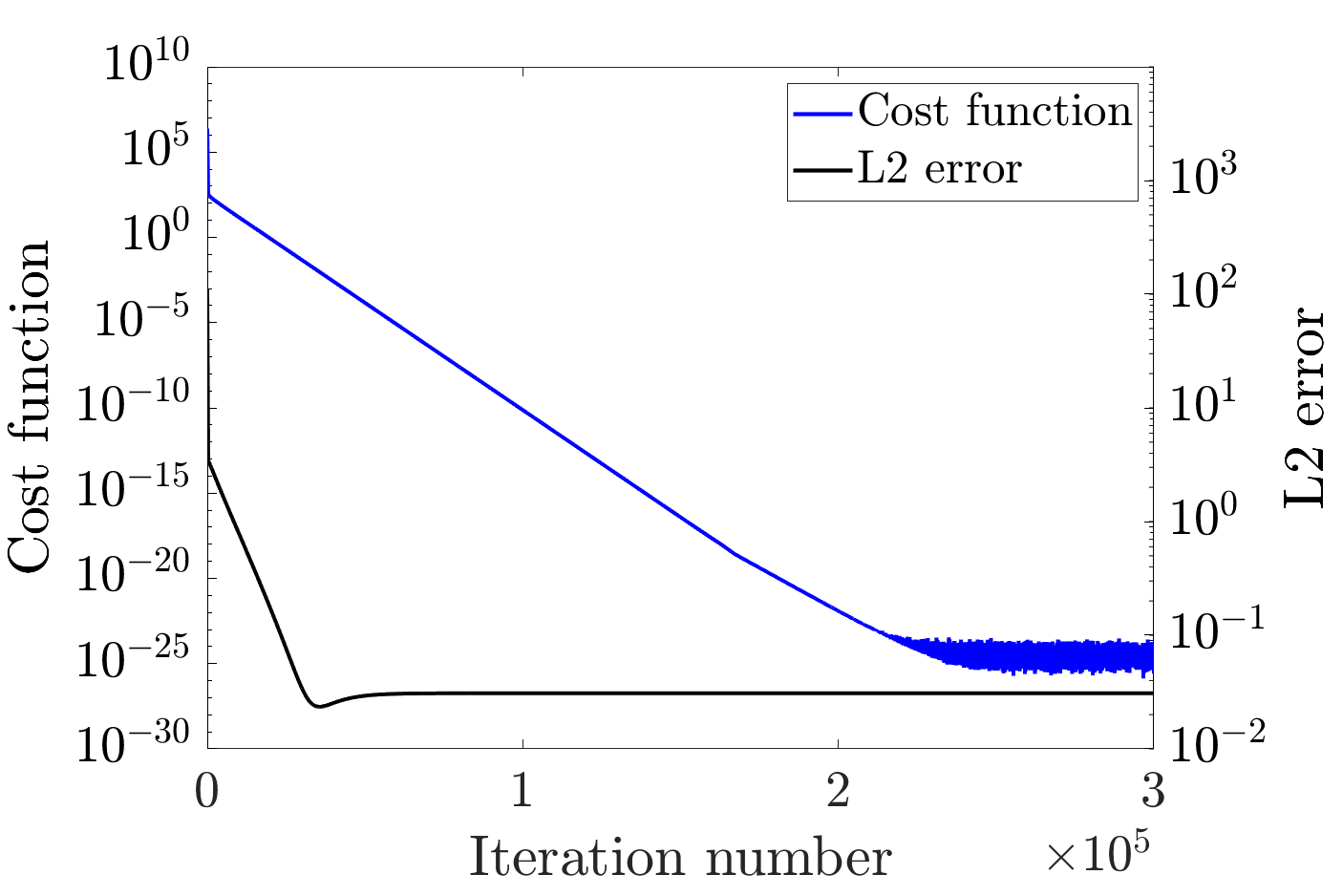}
        \begin{caption}{\label{fig2}$N=20$, $N_{\rm iter}=3\times10^5$, $\eta=10^{-6}$, $\beta=0$.}
        \end{caption}
    \end{center}
\end{figure}
In Figure \ref{fig2} we have plotted on the left the approximation we have obtained and on the right the
values of the cost function and the error in the neural network, as the number of iterations increases.
The experiment is done for $N=20$ and we can observe (see the vertical lines on the left picture corresponding
to values of the nodes $x_j=j/20$, $j=1,\ldots, 19$) that the method gives an approximation with small error at the nodes of
the partition $\tau_{20}=\left\{x_j\right\}_{j=1}^{19}$. As expected, the approximation is piecewise linear
but over a mesh different from $\tau_{20}$, which means that the neural network has modified (at least some) of the initial given
values for $W^{[2]}$ and $b^{[2]}$. The fact that the neural network is accurate over $\tau_{20}$ but
has a big error in the extra nodes makes sense since the {cost function is defined over} the nodes of $\tau_{20}$ and does not have any added information about other extra nodes. Let us observe that the cost function \eqref{cost_1}
is blind for the nodes not belonging to $\tau_{20}$ which means that the spurious oscillations of the neural network
approximation at the extra nodes cannot be avoided with this algorithm.
At this point it would be interesting to be able to add some information to the neural network (cost function) {in order} to produce a better approximation than $u_h$.
Taking into account that the response $F(x)$ is a piecewise linear approximation over a mesh with $3(N-1)$ nodes while $u_h$ is based on $N-1$ this could be, in principle, possible.
However, our aim in this experiment is to check {whether} a neural network for which the linear finite element approximation is a minimum of the cost function is able by itself to reproduce (learn)
the linear finite element approximation.

In Figure \ref{fig4} we have represented on the left the approximation corresponding to $N=40$ with
a smaller value of the learning rate $\eta=10^{-7}$. As before, we mark with vertical lines the
nodes of the uniform partition $\tau_{40}$, at which the neural network gives a small error. On the
right of Figure \ref{fig4} we have represented the approximation for $N=100$. Spurious oscillations
appear in both approximations with smaller local maximum and minimum values in the last
case: $N=100$.

\begin{figure}[h]
    \begin{center}
        \includegraphics[height=4.5truecm]{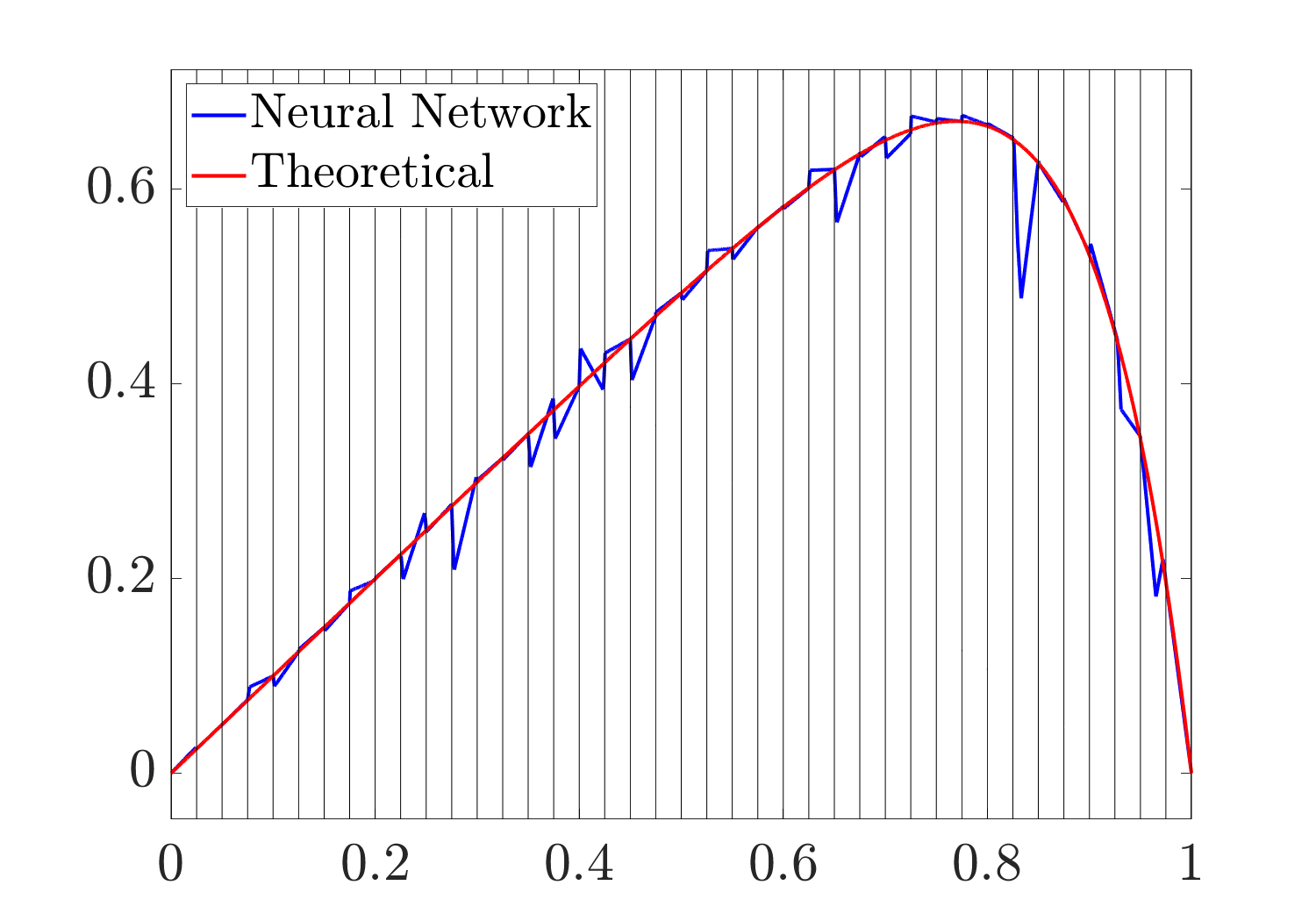}\
        \includegraphics[height=4.5truecm]{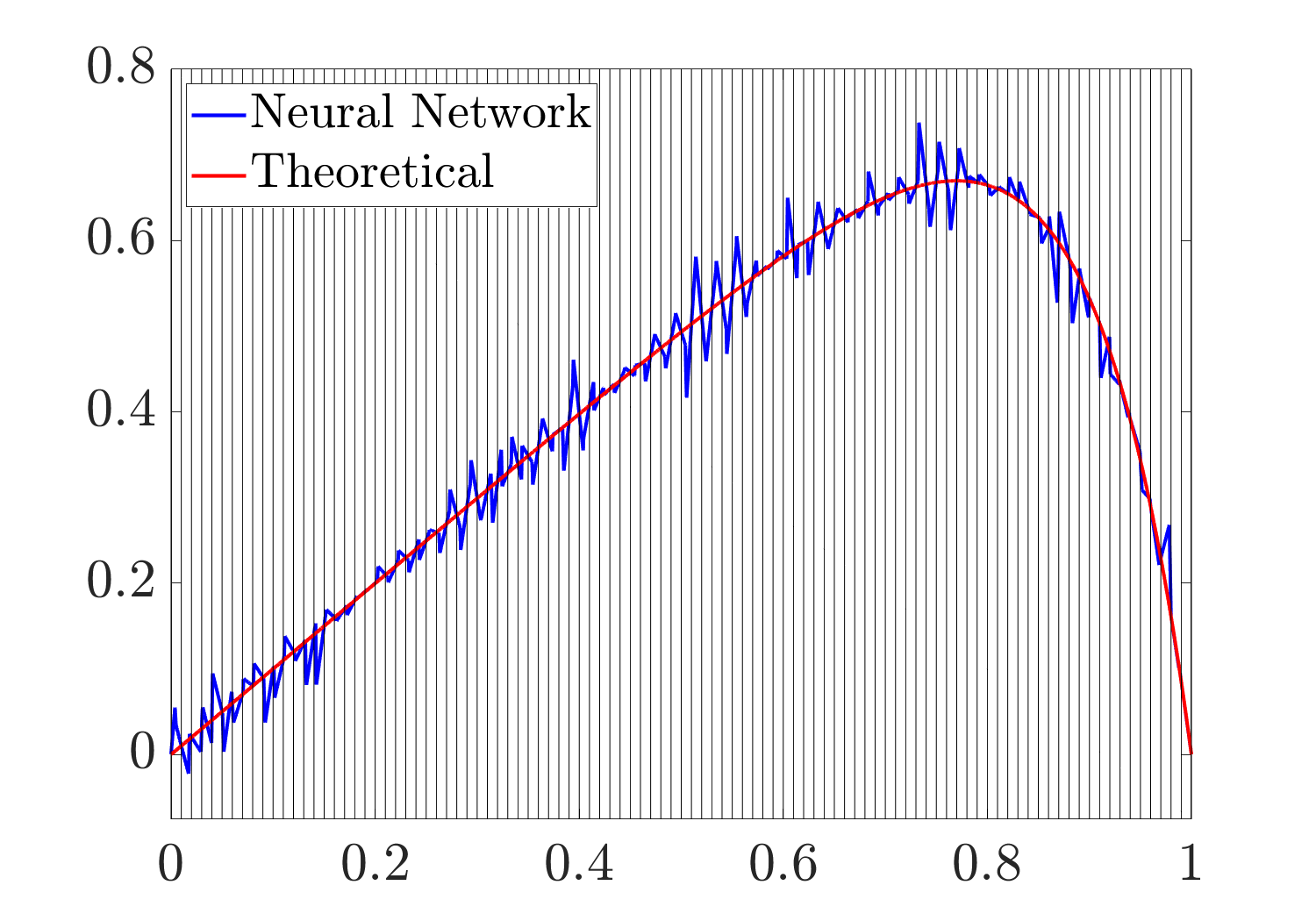}
        \begin{caption}{\label{fig4}
            $N=40$, $N_{\rm iter}=5\times10^5$, $\eta=10^{-7}$ and $\beta=0$ on the left and
            $N=100$, $N_{\rm iter}=3\times10^5$, $\eta=10^{-8}$ and $\beta=0$ on the right.}
        \end{caption}
    \end{center}
\end{figure}
In our third experiment we try to lead the network to the linear finite element approximation in a stronger way. To this end, we fix the values of $W^{[2]}$ and $b^{[2]}$ to those in \eqref{W2b2} and keep only $3(N-1)$ free parameters, those in $W^{[3]}$. We observe that in this case our neural network can be written as
\begin{equation}\label{fix}
    F(x)=\sum_{i=1}^{N-1}w^{3,1}_i\sigma\left(\frac{x-x_{i-1}}{h_{i-1}}\right)+w^{3,2}_i\sigma\left((x - x_{i}) \left( \frac{1}{h_{i-1}} + \frac{1}{h_{i}} \right)\right)+w^{3,3}_i\sigma\left(\frac{x - x_{i+1}}{h_{i}}\right).
\end{equation}
This means that the response of the network is a continuous piecewise linear function with respect to the original partition $\tau_N$ that satisfies in a strong way the boundary condition $F(0)$. In our numerical experiments we have plotted the resulting functions $F(x)$ at different stages of the iterative algorithm (different values of the number
of iterations) and we have observed that the main effort of the neural network seems to be to enforce the condition $F(1)=0$.
\begin{figure}[h]
    \begin{center}
        \includegraphics[height=4.5truecm]{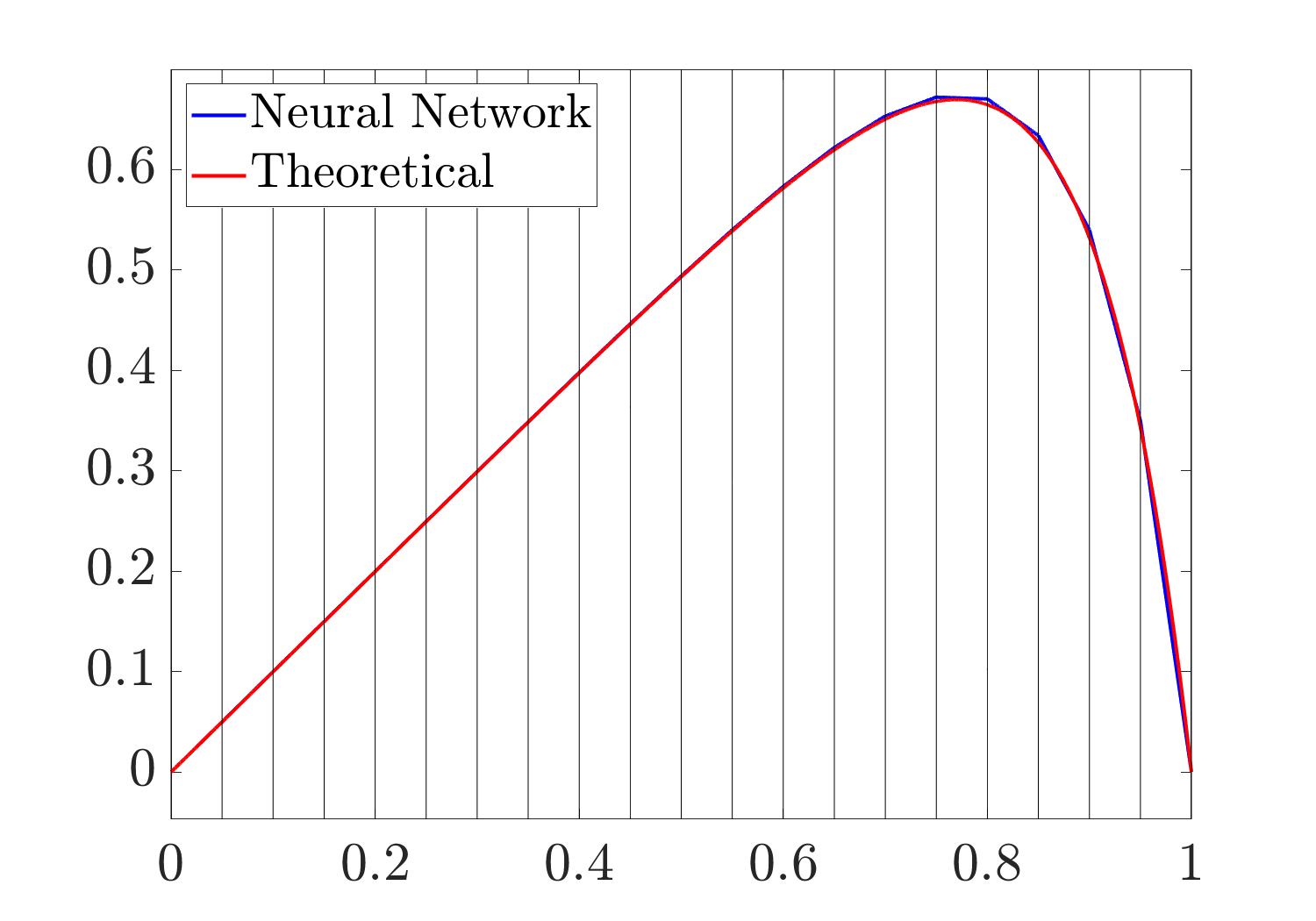}\
        \includegraphics[height=4.5truecm]{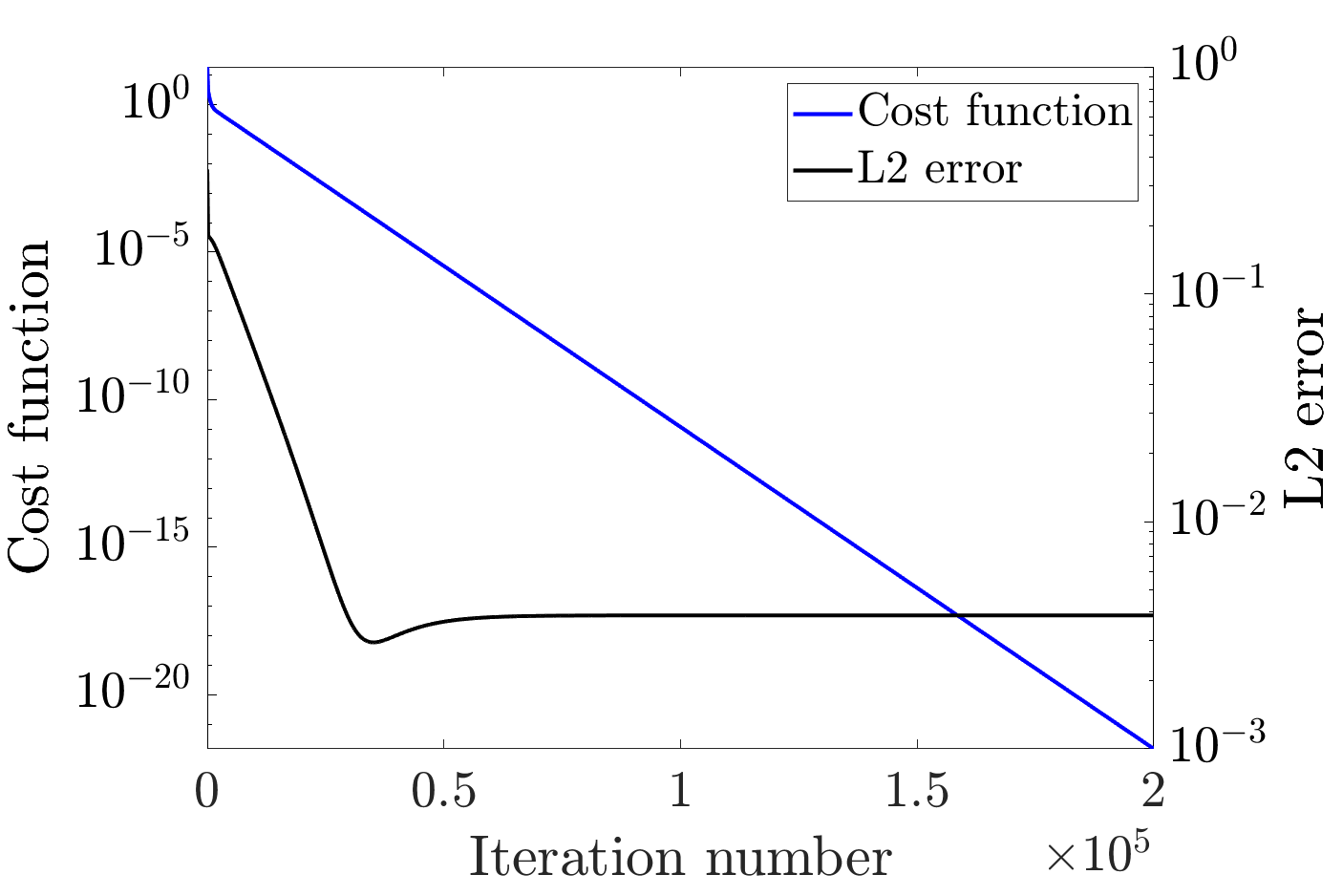}
        \begin{caption}{\label{fig5}$N=20$, $N_{\rm iter}=2\times10^5$, $\eta=10^{-6}$, $\beta=0$.}
        \end{caption}
    \end{center}
\end{figure}

In Figure \ref{fig5} we can see the neural network approximation on the left and the cost function and error
of the approximation on the right for $N=20$. In this case the final
approximation generated by the neural network is a good approximation and essentially indistinguishable from the
linear finite element approximation $u_h$. We have
obtained similar results for $N=40$ and $N=100$. In Figure \ref{fig6} we plot the absolute value of the difference between the linear finite element and the neural network approximations for $N=20$, $40$ and $100$. We observe that in all
cases the error
between both approximations increases from a point around $x=0.6$ and has a maximum at $x=1$. The errors
at $x=1$ are around $10^{-11},$ $10^{-8}$ and $10^{-3}$ for $N=20$, $40$ and $100$, respectively.

\begin{figure}[h]
    \begin{center}
        \includegraphics[height=3truecm]{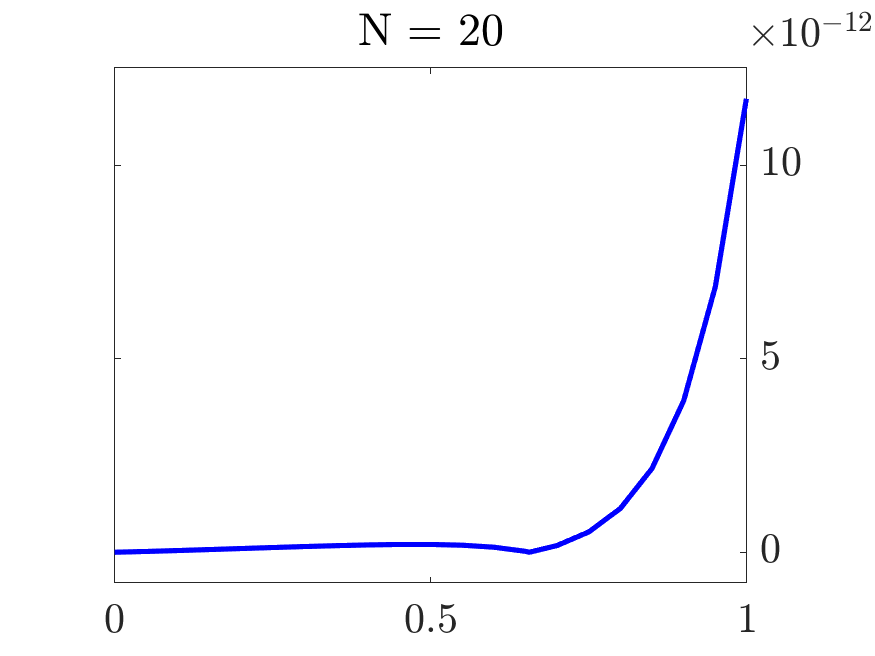}\
        \includegraphics[height=3truecm]{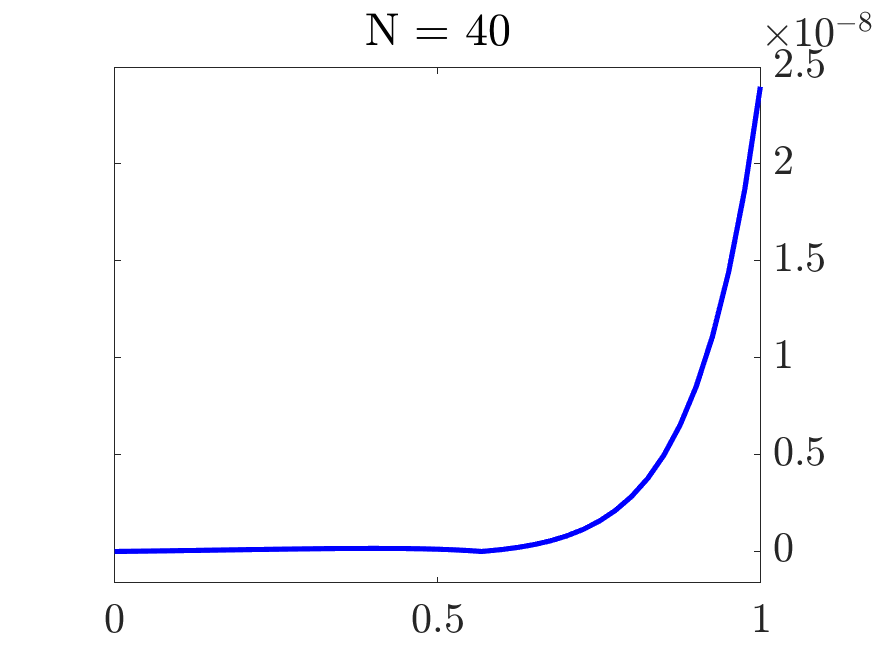}\
        \includegraphics[height=3truecm]{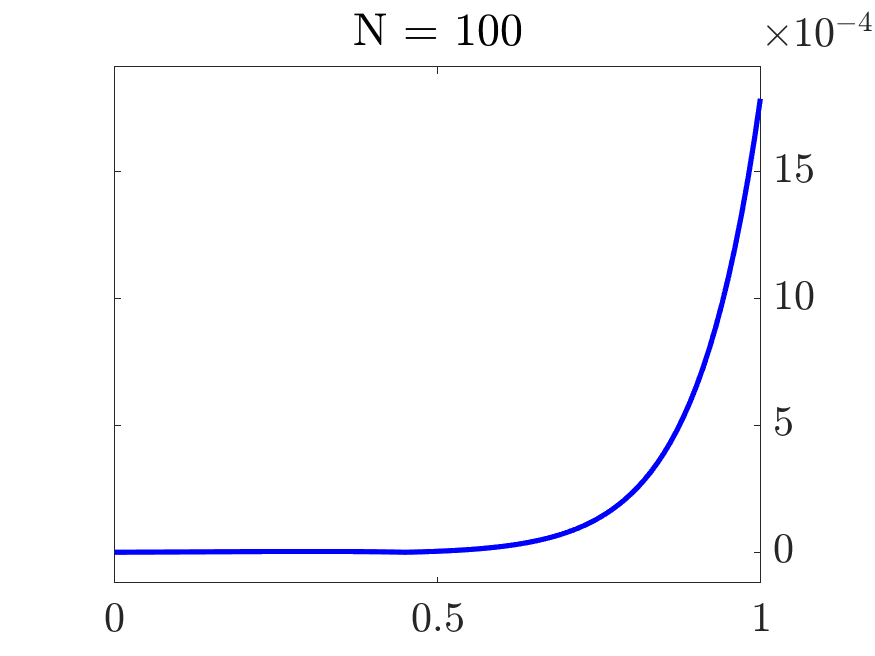}
        \begin{caption}{\label{fig6}Absolute value of errors between finite element and neural network approximations
                for $N=20, 40$ and $100$.}
        \end{caption}
    \end{center}
\end{figure}

\begin{table}[h]
    \begin{center}
        \begin{tabular}{|c|c|c|c|c|}
            \hline
                    & $L^2$ {\rm FEM} & $H^1$ {\rm FEM} & $L^2$ {\rm NN} & $H^1$ {\rm NN} \\ \hline
            $N=20$  & $3.87e-3$       & $3.20e-1$       & $3.87e-3$      & $3.20e-1$      \\
            \hline
            $N=40$  & $9.72e-4$       & $1.61e-1$       & $9.72e-4$      & $1.61e-1$      \\
            \hline
            $N=100$ & $1.56e-4$       & $6.45e-2$       & $3.30e-4$      & $2.60e-1$      \\
            \hline
        \end{tabular}
    \end{center}
    \caption{Errors in the finite element and neural network approximations.}
    \label{tabla1}
\end{table}

As we can also observe in Table \ref{tabla1}, in which we have represented the errors in the finite element and neural network approximations, both approximations give the same errors for $N=20$ and $N=40$, both in $L^2$ and $H^1$ norms, while for $N=100$ the finite element approximation produces smaller errors. In our experiments we observed that for $N=100$, both the cost function and the error are decreasing functions (they do not stabilize). We
stopped the algorithm after $5\times 10^5$ iterations.
We guess that by increasing the number of iterations, the network
could had finally reached the error of the finite element approximation as in the
previous cases: $N=20$ and $N=40$.
Our aim is, however,  not to reproduce the linear finite element approximation that can be obtained solving the linear system described in \eqref{eq_gal}-\eqref{coef} but find out how much information needs an a priori blind neural network  with a cost function for which the linear finite element is a minimum to learn (or reproduce) this
approximation. Since, in general, one does not even know the structure of the neural network, with this simple experiment we just want to  stand out the problems that Physical Informed Neural Networks (or similarly based networks) may have in obtaining good approximations to partial differential equations.

\begin{table}[h]
    \begin{center}
        \begin{tabular}{|c|c|c|c|c|}
            \hline
                    & $L^2$ {\rm SUPG} & $H^1$ {\rm SUPG} & $L^2$ {\rm NN} & $H^1$ {\rm NN} \\ \hline
            $N=20$  & $1.25e-1$        & $21.91$          & $1.25e-1$      & $21.91$        \\
            \hline
            $N=40$  & $8.52e-2$        & $21.45$          & $5.79e-2$      & $42.67$        \\
            \hline
            $N=100$ & $4.87e-2$        & $20.04$          & $2.18e-2$      & $134.37$        \\
            \hline
        \end{tabular}
    \end{center}
    \caption{Errors in the SUPG finite element and neural network approximations.}
    \label{tabla2}
\end{table}

\subsection{Convection dominated regime}
We have reproduced the previous experiments in the convection dominated regime.
In this section  we take $\epsilon=0.001$ and, as before, we consider a uniform partition \eqref{partition} for different values of $N$.
In this case, instead of \eqref{cost_1} we use the cost function \eqref{cost_2}.
In view of the previous results we did not expect to get much different results (probably worse) in this more difficult problem.
\begin{figure}[h]
    \begin{center}
        \includegraphics[height=4.5truecm]{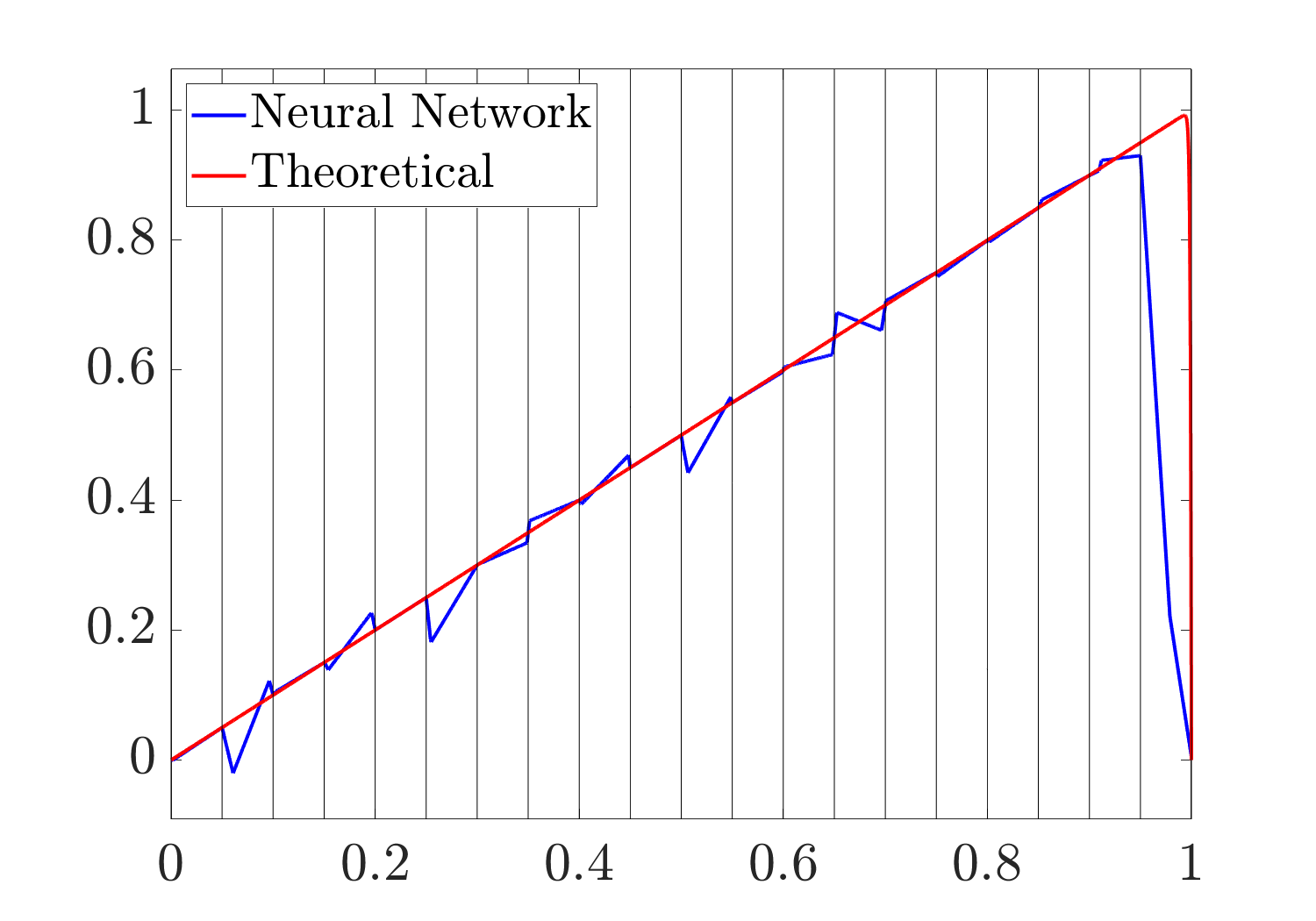}\
        \includegraphics[height=4.5truecm]{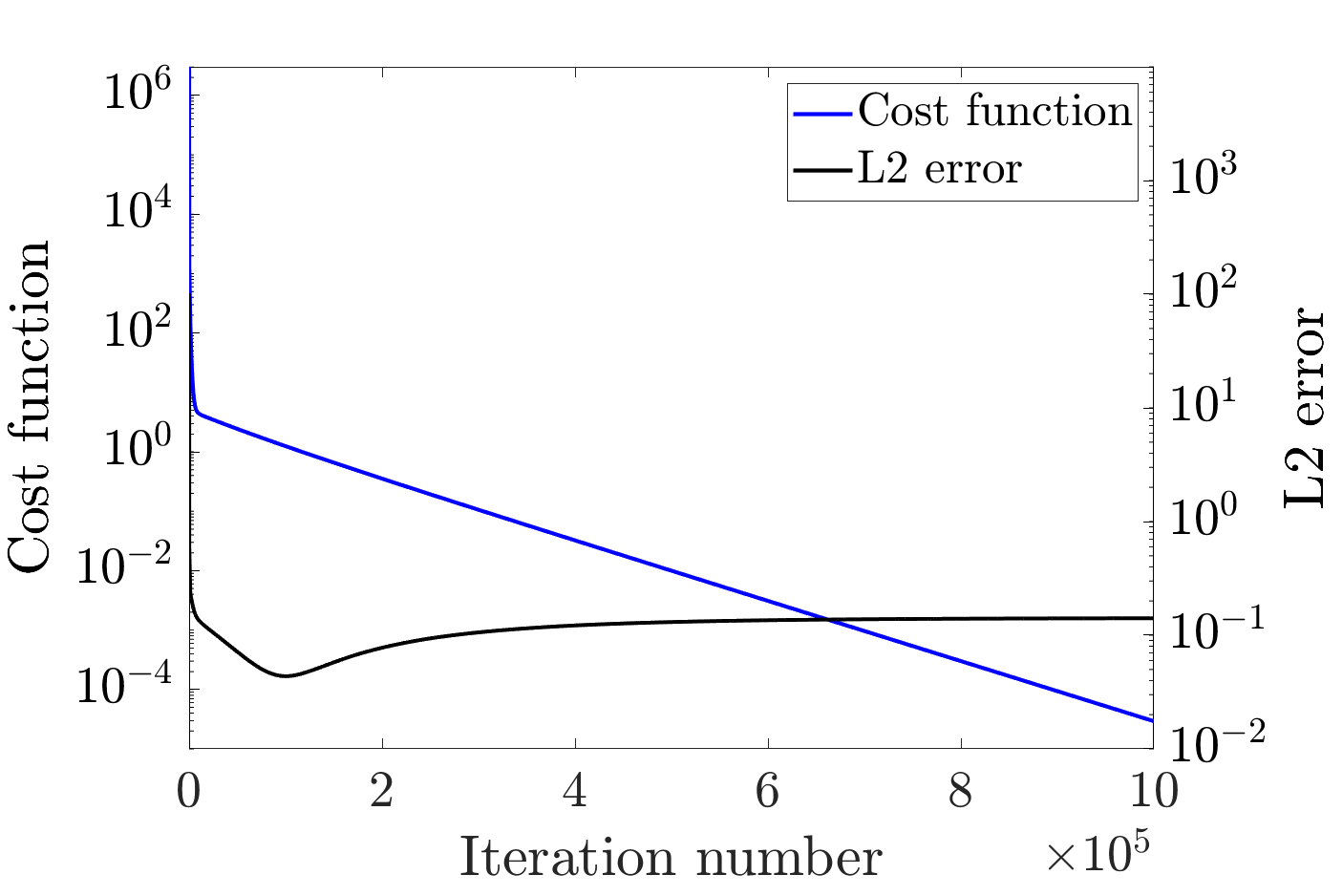}
        \begin{caption}{\label{fig7}$N=20$, $N_{\rm iter}=10^6$, $\eta=10^{-6}$, $\beta=0$.}
        \end{caption}
    \end{center}
\end{figure}
\begin{figure}[h]
    \begin{center}
        \includegraphics[height=4.5truecm]{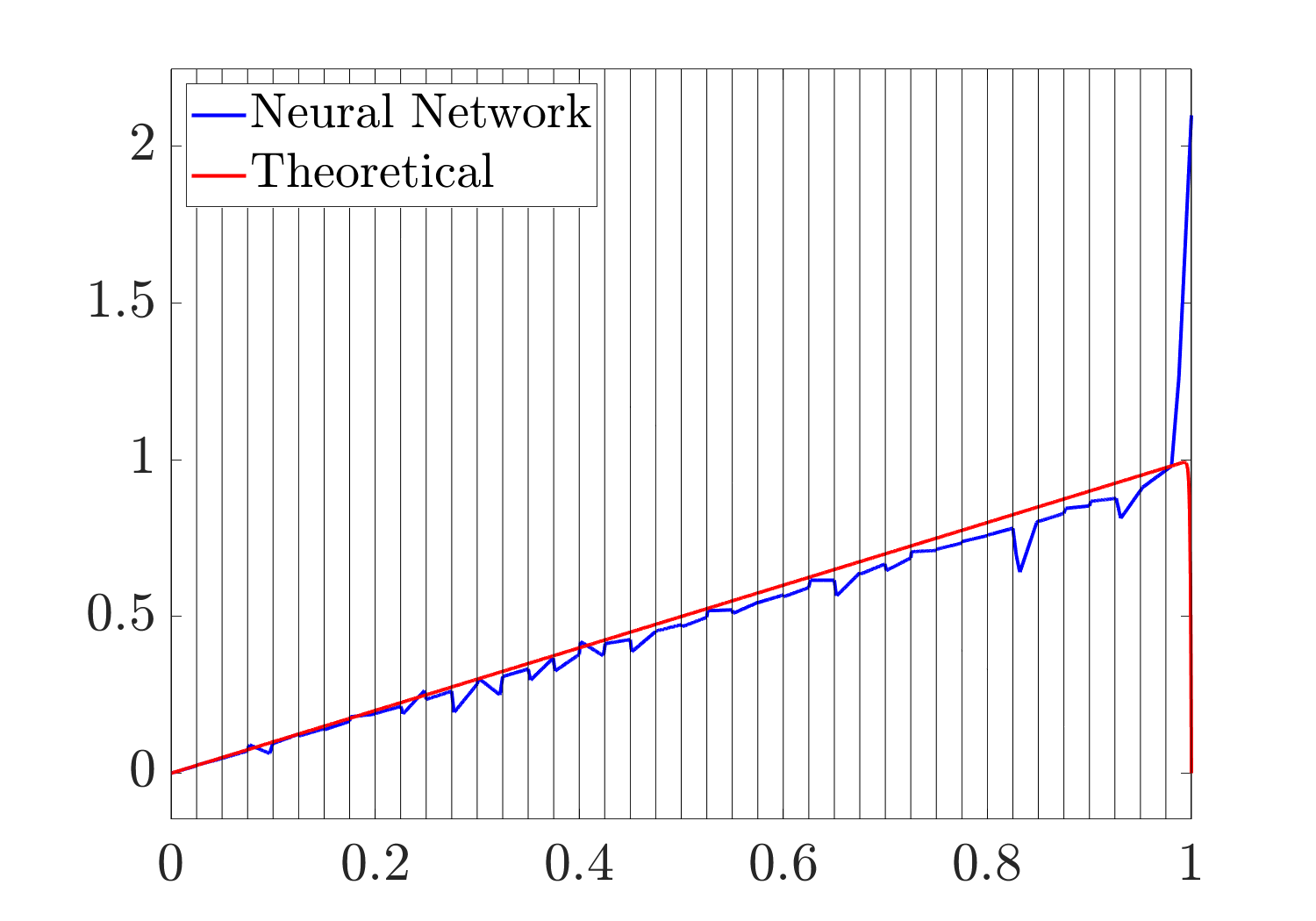}\
        \includegraphics[height=4.5truecm]{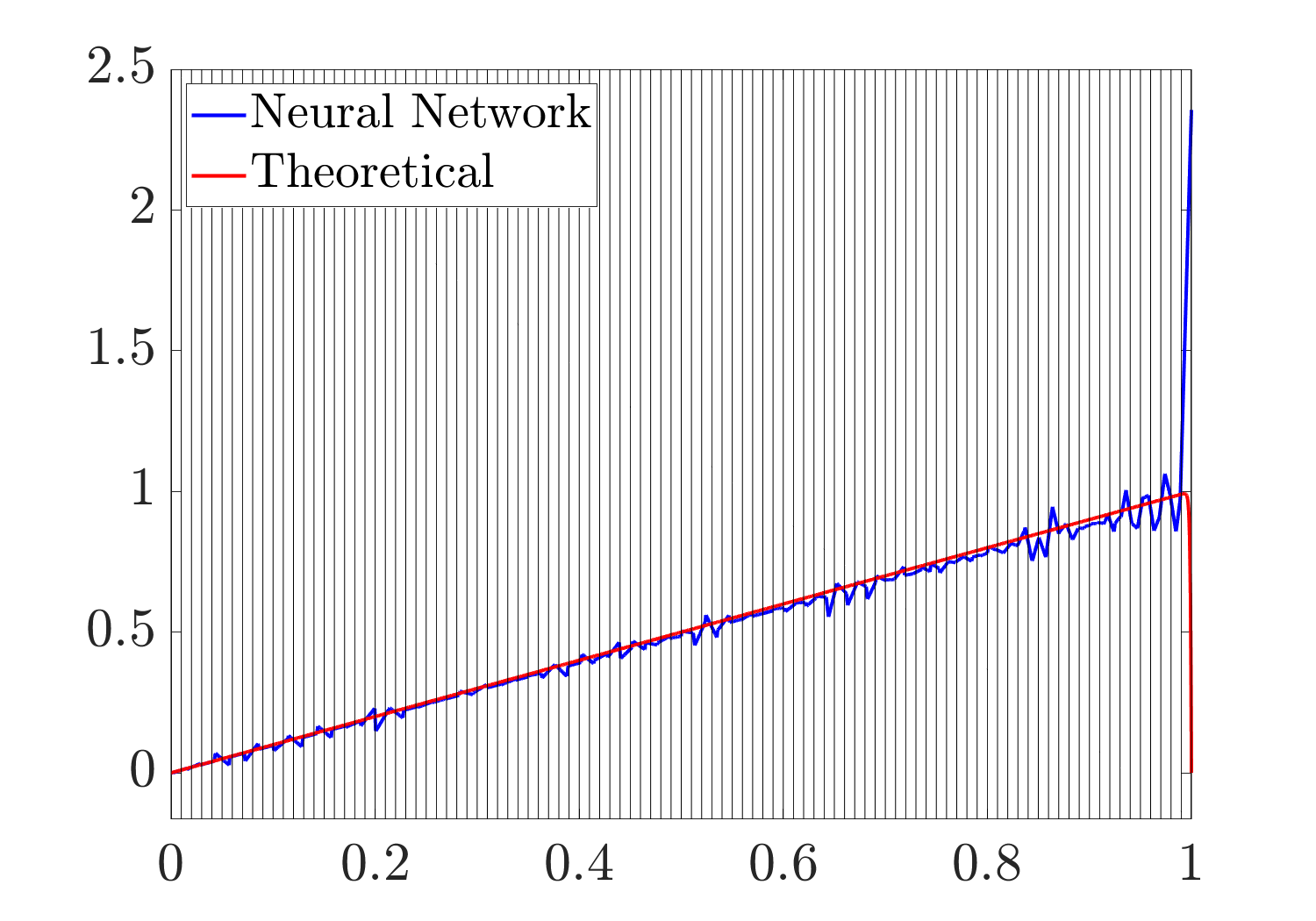}
        \begin{caption}{
            \label{figura7}
            $N=40$, $N_{\rm iter}=10^6$, $\eta=10^{-7}$ and $\beta=0$ on the left and\\
            $N=100$, $N_{\rm iter}=10^6$, $\eta=10^{-8}$ and $\beta=0$ on the right.
            }
        \end{caption}
    \end{center}
\end{figure}

\begin{figure}[h]
    \begin{center}
        \includegraphics[height=4.5truecm]{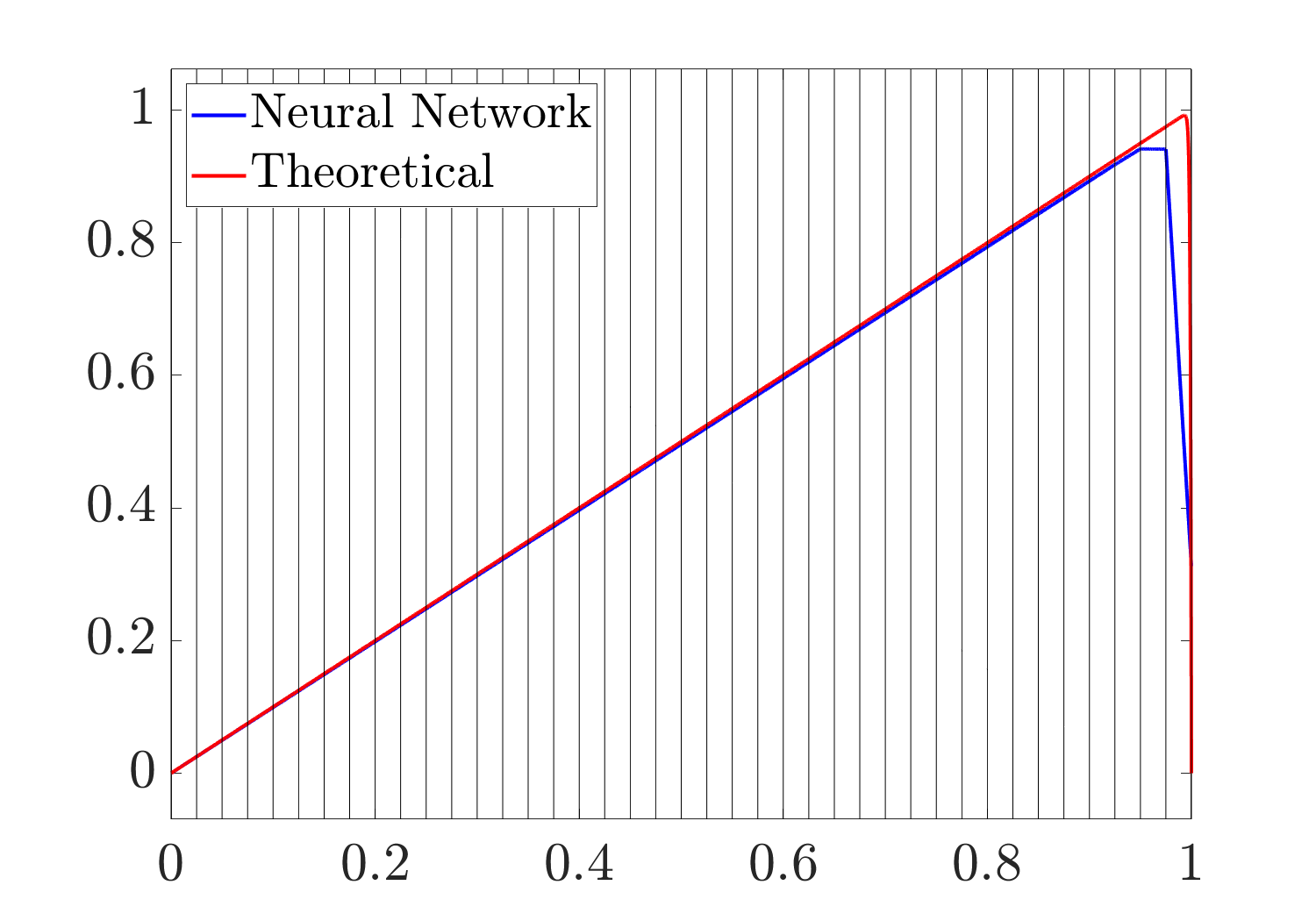}\
        \includegraphics[height=4.5truecm]{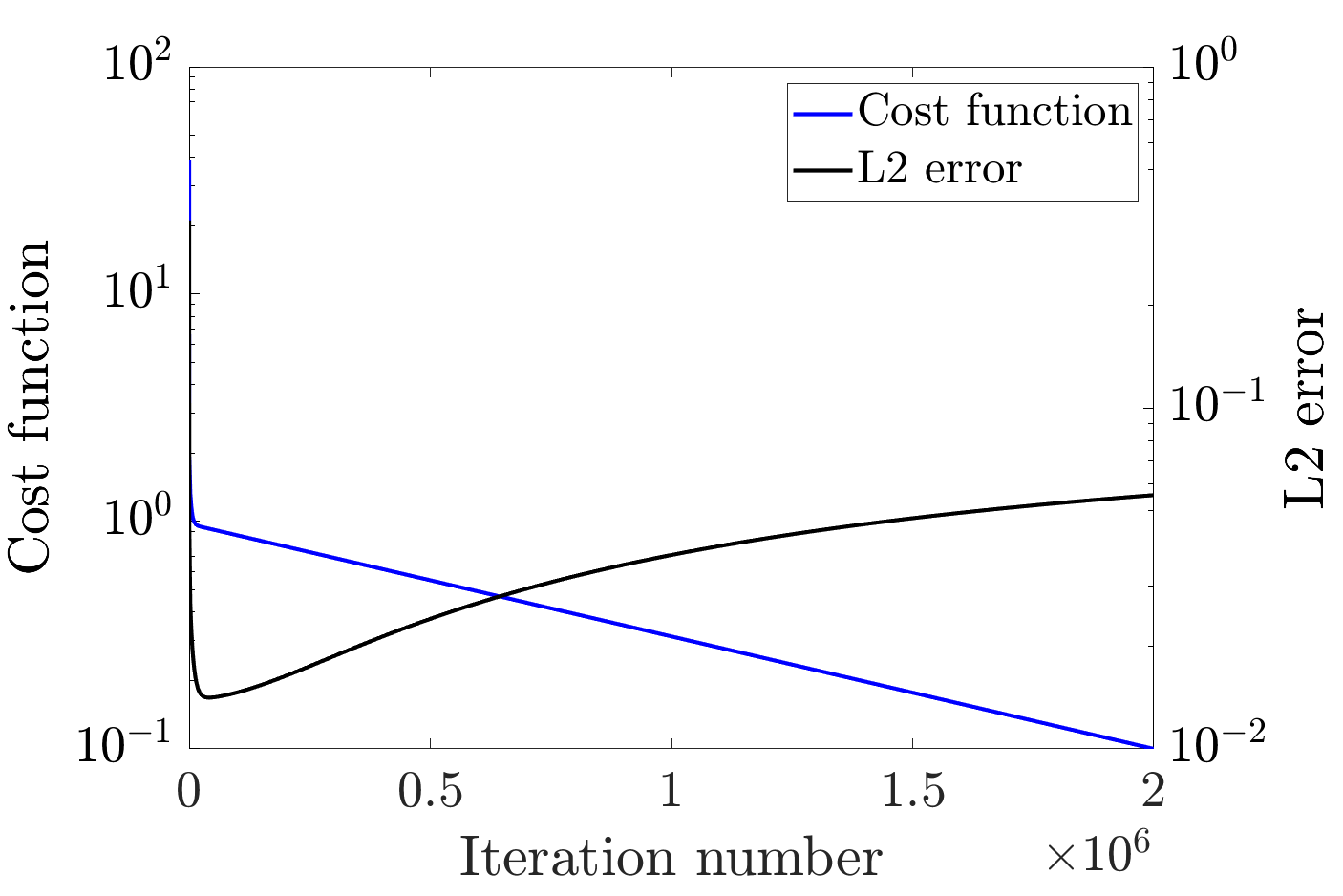}
        \begin{caption}{\label{fig8}$N=40$, $N_{\rm iter}=2 \times 10^6$, $\eta=10^{-7}$, $\beta=0$.}
        \end{caption}
    \end{center}
\end{figure}
\begin{figure}[h]
    \begin{center}
        \includegraphics[height=4.5truecm]{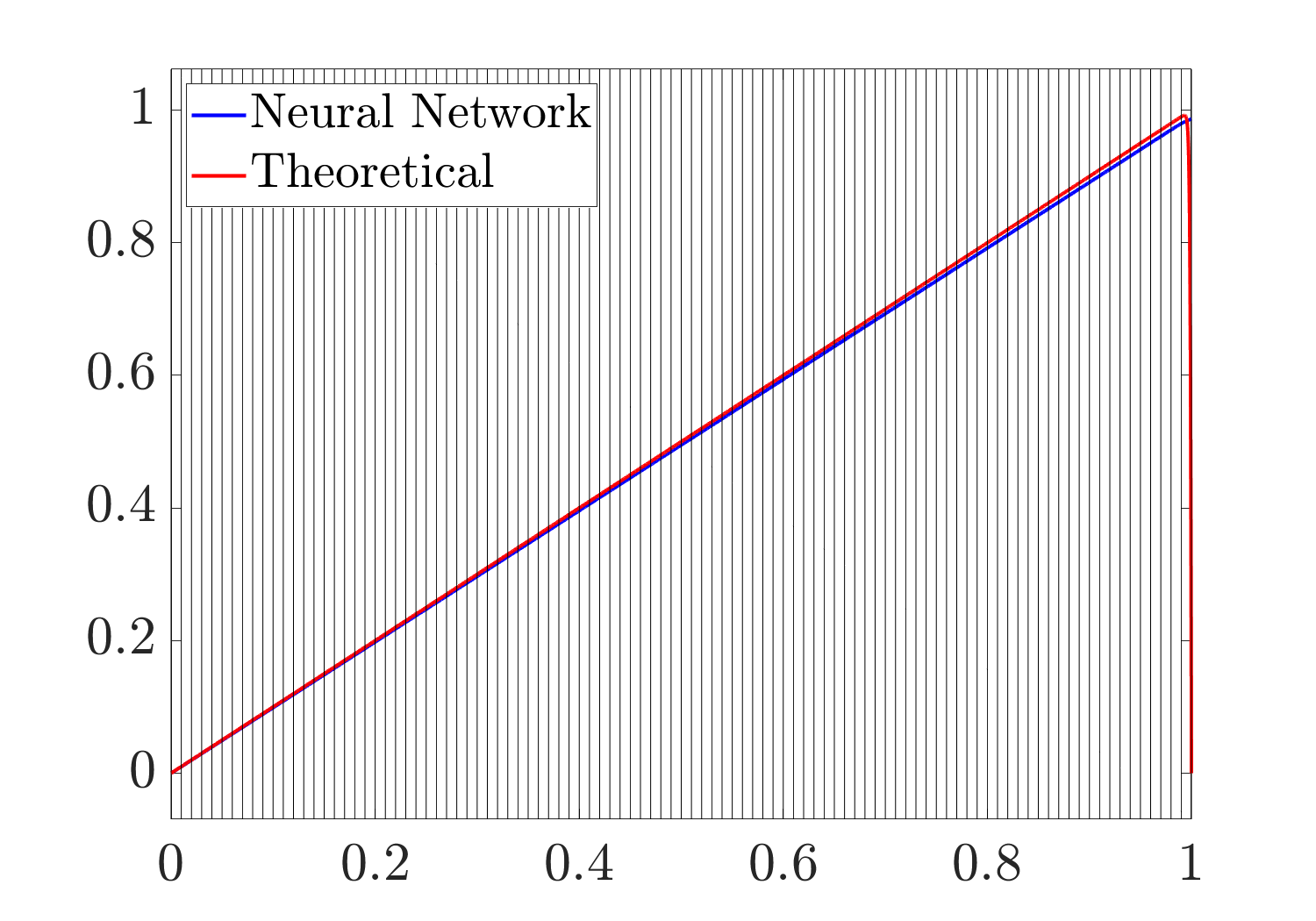}\
        \includegraphics[height=4.5truecm]{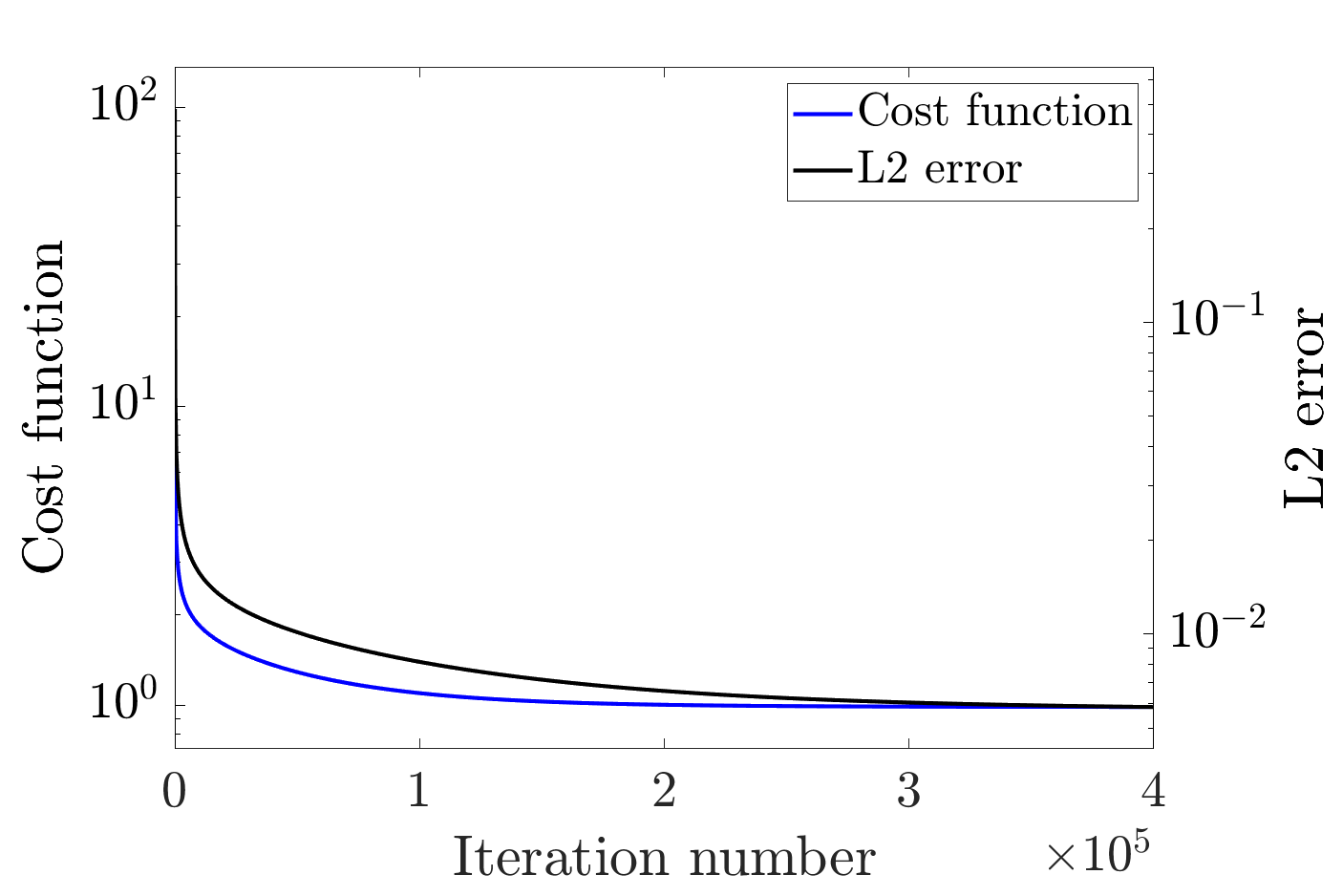}
        \begin{caption}{\label{fig9}$N=100$, $N_{\rm iter}=4 \times 10^5$, $\eta=10^{-9}$, $\beta=0$.}
        \end{caption}
    \end{center}
\end{figure}

As before, we do not obtain a reasonable approximation for the case in which all the weights and biases are free. In the second case, in which we initialize $W^{[2]}$ and $b^{[2]}$ as in \eqref{W2b2}, but then  we allow the neural network to change those values, we obtained similar results  with some peculiarities. The case $N=20$ is indeed very similar. As we can observe in Figure \ref{fig7},
the neural network approximation is a piecewise linear approximation over a refined partition of $\tau_{20}$. Again, the neural network approximation is only a good approximation to the true solution in the nodes of $\tau_{20}$.

In Figure \ref{figura7} we have represented on the left the neural network approximation for $N=40$ and on the right
the one corresponding to $N=100$. Here, we find some differences with respect to the diffusion dominated regime. We observe that in this case the errors at the nodes of the underlying uniform
partition are larger than in the diffusion dominated case and also that the error at $x=1$ is very large. This means that the neural network fails at imposing the condition $F(1)=0$. As it is known, there is a boundary layer at $x=1$ in this extremely simple convection dominated problem.

To conclude, we fix $W^{[2]}$ and $b^{[2]}$ as in \eqref{W2b2}
and leave $W^{[3]}$ free to check if the network is able to reproduce the SUPG linear finite element approximation.
In Table \ref{tabla2} we have collected the errors of the SUPG and neural network approximations.
As expected, the errors are larger than those in Table \ref{tabla1} for the diffusion dominated case, in particular
all the errors in $H^1$ are considerably large.
While for $N=20$ the network is able to reproduce the SUPG approximation, similarly to how it reproduced the FE approximation in the diffusion dominated case (compare with Table \ref{tabla1}),
the situation changes for $N=40$ and $N=100$.
For these values of $N$, the $L^2$ errors are smaller in the neural network approximation,
while the $H^1$ errors are considerably larger.
We believe that the explanation for this is the weak imposition of the right homogeneous Dirichlet boundary condition at $x=1$.
While in the SUPG method this condition is strongly imposed, in the neural network is not.
In fact, the neural network approximations fail to impose the condition $F(1) = 0$ within the specified number of iterations in both cases.
To study this in more detail, we have represented the results of the experiments in Figures \ref{fig8} and \ref{fig9}.
It is surprising to see that, although the approximations are quite accurate
(except for the values at $x=1$ being far from 0 in both cases), the values of the cost function are large.
For $N=40$ (Figure \ref{fig8}) the cost has not yet stabilized, hence we believe that with enough iterations the
neural network would have eventually approached the SUPG approximation, leading to a further increase in the $L^2$ error.
For $N=100$ (Figure \ref{fig9}), the value of the cost function stabilizes at around $1$,
meaning that the neural network is not able to reproduce the SUPG approximation.
However, it yields a better approximation in terms of the $L^2$ error
than SUPG, even though the value of the cost function is large.
Since the limit solution ($\epsilon=0$) of problem \eqref{eq_model} is the function $G(x)=x$, which satisfies only the left
boundary condition, the reason for having a boundary layer at $x=1$ is the strong imposition of the right boundary condition.
As we can see in the left picture of Figure \ref{fig9}, the neural network approximation for $N=100$ is close to the function $G(x)=x$.
On the other hand, $G(x)=x$ is also close in $L^2$ to the solution of problem \eqref{eq_model} for the selected value of $\epsilon$.
Observe that function $G(x)=x$ can be expressed as a neural network of the form \eqref{fix} by taking 
all parameters equal to zero except for $w^{3,1}_1=h$.
Actually, after the training process, the neural network presents a value of $w^{3,1}_1=0.0098 \approx h=0.01$,
and most of the weights in $W^{[3]}$ are close to zero ($< 10^{-10}$).
Thus, in this particular example, the weak imposition of the right boundary condition together with the particular form of the
output of the neural network seems to make the neural network approach the limit solution of \eqref{eq_model}.

\section{Conclusions}
In this note we have constructed a neural network for which the linear finite element approximation of a simple one dimensional boundary value problem is a minimum of the cost function to find out if the neural network is able to reproduce the finite element approximation.
In the first attempt, the optimization problem entails too many free parameters, rendering it highly over-determined and
making it difficult for the network to find the optimal values.
In this case, the approximation is far away from the linear finite element approximation and the theoretical solution.
In subsequent steps, we shift the focus into finding out how much information an a priori blind neural network
with a cost function for which the linear finite element is a minimum needs to correctly learn this approximation.
With this simple example we just want to stand out the problems that PINNs (or similarly based networks) may have in obtaining good
approximations to partial differential equations since, in general, one does not even know the structure of the neural network.
Our conclusion is that in our simple model problem we do not obtain, in general, better approximations than the finite element method.
In fact, we only obtain neural network approximations close to those of the finite element method in our last try,
where we have fixed so many weights and biases that the response of the neural network, regardless of the output values of the optimization problem,
is always a continuous piecewise linear approximation with respect to the partition in which the linear finite element approximation is defined. 

Our conclusions are in agreement with the deeper study of \cite{schonlieb_et_al} in which the authors compare
PINNs and finite elements through various linear and nonlinear partial differential equations.
Their study suggests that for certain classes of PDEs for which classical methods are applicable, PINNs are not able to outperform these.
However, as stated in \cite{schonlieb_et_al}, we believe that PINNs could be efficient in high-dimensional problems for which classical techniques are prohibitively expensive and when combining both PDEs and data.

\bigskip

\noindent {\bf Data availability}
\bigskip

\noindent
The MATLAB code for the training algorithms can be found in the following GitHub repository:
\href{https://github.com/EduardoTerres/Can-Neural-Networks-learn-Finite-Elements.git}
{https://github.com/EduardoTerres/Can-Neural-Networks-learn-Finite-Elements.git}


\begin{thebibliography}{5}
    \bibitem{BH} A. N. Brooks \& T. J. R. Hughes, {\it Streamline upwind/Petrov-Galerkin formulations for convection dominated flows with particular emphasis on the incompressible Navier-Stokes equations}, Comput. Methods Appl. Mech. Engrg. 32, 1982, 199-259.
    \bibitem{GanderWanner2012} M. J. Gander and G. Wanner, {\it From Euler, Ritz, and Galerkin to Modern Computing}, SIAM Review 54, 2012, 627-666.
    \bibitem{schonlieb_et_al} T. G.Grossmann, U.J. Komorowska, J. Latz \& C.-B. Schonlieb, {\it Can Physics-Informed Neural Networks  beat the finite element method?} 2023. 10.48550/arXiv.2302.04107.
    \bibitem{He_et_al} J. He, L. Li, J. Xu \& Ch. Zheng, {\it ReLU Deep Neural Networks and Linear Finite elements}, J. Comp. Math. 38, 2020, 502--527.
    \bibitem{siam_rev1} C. F. Higham \& D. J. Higham, {\it Deep learning: an Introduction for Applied Mathematicians},  SIAM Review, 61, 2019, 860--891.
    \bibitem{mis_mo} S. Mishra \& R. Molinaro, {\it Estimates of the generalization error of physics-informed neural networks for approximating PDEs}, IMA J. Numer. Anal. 43 2023, 1--43.
    \bibitem{siam_rev2} L. Lu, X. Meng, Z. Mao \& G. E. Karniadakis, {\it DeepXDE: A Deep Learning Library for Solving Differential Equations}, SIAM Review, 63, 2021, 208--228.
    \bibitem{Pinkus} A. Pinkus, {\it Approximation theory of the MLP model in neural networks}, Acta Numerica 8, 1999, 143--195.
    
\end{thebibliography}
\end{document}